\documentclass{amsart}

\usepackage{amsmath, amssymb, amsthm, amscd}

\theoremstyle{plain}
\newtheorem{theorem}{Theorem}[section]

\newtheorem{lemma}[theorem]{Lemma}

\theoremstyle{definition}
\newtheorem{definition}[theorem]{Definition}
\newtheorem{example}[theorem]{Example}

\theoremstyle{remark}
\newtheorem{remark}[theorem]{Remark}

\newcommand{\N}{\mathbb{N}}                   
\newcommand{\dl}{\delta}                      
\newcommand{\dlp}{\dl^+}                      
\newcommand{\dlm}{\dl^-}                      
\newcommand{\M}{(\mathcal{M})}
\newcommand{\A}{(\mathcal{A})}

\numberwithin{equation}{section}

\begin{document}

\title{On the Meyniel condition for hamiltonicity in bipartite digraphs}
\author{Janusz Adamus}
\address{J.Adamus, Department of Mathematics, The University of Western Ontario, London, Ontario N6A 5B7 Canada
         - and - Institute of Mathematics, Faculty of Mathematics and Computer Science, Jagiellonian University,
         ul. {\L}ojasiewicza 6, 30-348 Krakow, Poland}
\email{jadamus@uwo.ca}
\author{Lech Adamus}
\address{L. Adamus, AGH University of Science and Technology, Faculty of Applied Mathematics, al. Mickiewicza 30, 30-059 Krakow, Poland}
\email{adamus@agh.edu.pl}
\thanks{The authors' research was partially supported by Natural Sciences and Engineering Research Council of Canada (J. Adamus) and Polish Ministry of Science and Higher Education (L. Adamus).}
\subjclass[2000]{05C20, 05C38, 05C45}
\keywords{digraph, bipartite digraph, cycle, hamiltonicity, degree condition}

\begin{abstract}
We prove a sharp Meyniel-type criterion for hamiltonicity of a balanced bipartite digraph:
For $a\geq2$, a bipartite digraph $D$ with colour classes of cardinalities $a$ is hamiltonian if $d_D(u)+d_D(v)\geq3a+1$ whenever $uv\notin A(D)$ and $vu\notin A(D)$. As a consequence, we obtain a sharp sufficient condition for hamiltonicity in terms of the minimal degree: a balanced bipartite digraph $D$ on $2a$ vertices is hamiltonian if $\dl(D)\geq(3a+1)/2$.
\end{abstract}
\maketitle


\section{Introduction}
\label{sec:intro}

The main goal of this article is to prove a Meyniel-type sufficient condition for hamiltonicity of a balanced bipartite digraph.
We consider digraphs in the sense of \cite{BT}, and use standard graph theoretical terminology and notation (see Section~\ref{subsec:not} for details).

Our object of study in the present article are bipartite digraphs satisfying the following Meyniel-type condition (cf. Thm.\,\ref{thm:Mey}).

\begin{definition}
\label{def:3a+1}
Consider a balanced bipartite digraph $D$ with colour classes of cardinalities $a$. We will say that $D$ satisfies \emph{condition $\M$} when
\[
d_D(u)+d_D(v)\geq 3a+1
\]
for every pair of distinct vertices $u,v\in V(D)$ such that $uv\notin A(D)$ and $vu\notin A(D)$.
\end{definition}

Our main result is the following:

\begin{theorem}
\label{thm:main}
Let $D$ be a balanced bipartite digraph with colour classes of cardinalities $a$, where $a\geq2$.
If $D$ satisfies condition $\M$, then $D$ contains an oriented cycle of length $2a$.
\end{theorem}

There are numerous sufficient conditions for existence of hamiltonian cycles in digraphs (see \cite{BT}). In this article, we will be concerned with the degree conditions. 
For general digraphs, let us recall the following four classical results.

\begin{theorem}[{Ghouila-Houri, 1960, \cite{G}}]
\label{thm:G}
Let $D$ be a strongly connected digraph on $n$ vertices, where $n\geq3$. If $\dl(D)\geq n$, then $D$ contains an oriented cycle of length $n$.
\end{theorem}

(A digraph is called \emph{strongly connected} when, for every (ordered) pair of distinct vertices $u$ and $v$, $D$ contains an oriented path originating in $u$ and terminating in $v$.)

\begin{theorem}[{Nash-Williams, 1969, \cite{NW}}]
\label{thm:NW}
Let $D$ be a digraph on $n$ vertices, where $n\geq3$. If $\dlp(D)\geq n/2$ and $\dlm(D)\geq n/2$, then $D$ contains an oriented cycle of length $n$.
\end{theorem}

\begin{theorem}[{Woodall, 1972, \cite{W}}]
\label{thm:W}
Let $D$ be a digraph on $n$ vertices, where $n\geq3$. If $d_D^+(u)+d_D^-(v)\geq n$ for every pair of distinct vertices $u,v\in V(D)$ satisfying $uv\notin A(D)$, then $D$ contains an oriented cycle of length $n$.
\end{theorem}

\begin{theorem}[{Meyniel, 1973, \cite{Mey}}]
\label{thm:Mey}
Let $D$ be a strongly connected digraph on $n$ vertices, where $n\geq3$. If $d_D(u)+d_D(v)\geq2n-1$ for any two vertices $u$ and $v$ such that $uv\notin A(D)$ and $vu\notin A(D)$, then $D$ contains an oriented cycle of length $n$.
\end{theorem}

All the above criteria are sharp (see~\cite{BT}). Note also that Theorems~\ref{thm:G}, \ref{thm:NW} and \ref{thm:W} follow from Theorem~\ref{thm:Mey}.
\medskip

Naturally, for bipartite digraphs one might expect bounds for degrees of order $|D|/2$ rather than $|D|$. This is the case, indeed, for analogues of the Nash-Williams and Woodall theorems. As for the analogues of the Ghouila-Houri and Meyniel theorems, however, this expectation is quite far from reality (cf. Remark~\ref{rem:at-least-3a}). For minimal half-degrees we have the following result.

\begin{theorem}[{Amar \& Manoussakis, 1990, \cite{AM}}]
\label{thm:AM}
Let $D$ be a bipartite digraph with colour classes $X$ and $Y$ such that $|X|=|Y|=a$, where $a\geq2$.
If $\dlp(D)\geq (a+2)/2$ and $\dlm(D)\geq(a+2)/2$, then $D$ contains an oriented cycle of length $2a$.
\end{theorem}

The above criterion is sharp. Moreover, it is shown in \cite{AM} that the only non-hamiltonian digraph $D$ satisfying $\dlp(D)\geq (a+1)/2$ and $\dlm(D)\geq(a+1)/2$ is the digraph on $6$ vertices depicted in Fig.\,1.

\begin{center}
\setlength{\unitlength}{.8cm}
\begin{picture}(7,3.5)
\put(1,1){\circle{0.2}}
\put(3,1){\circle{0.2}}
\put(5,1){\circle{0.2}}
\put(1,3){\circle*{0.2}}
\put(3,3){\circle*{0.2}}
\put(5,3){\circle*{0.2}}
\put(1,3){\vector(0,-1){1.9}}
\put(1,3){\vector(2,-1){3.9}}
\put(3,3){\vector(0,-1){1.9}}
\put(3,3){\vector(1,-1){1.9}}
\put(5,3){\vector(-2,-1){3.9}}
\put(5,3){\vector(-1,-1){1.9}}
\put(1.1,1.1){\vector(1,1){1.8}}
\put(3,2){\vector(2,1){1.9}}
\put(2.9,1.1){\vector(-1,1){1.8}}
\put(4.5,2.5){\vector(1,1){0.4}}
\put(3,2){\vector(-2,1){1.9}}
\put(3.5,2.5){\vector(-1,1){0.4}}
\put(3,0){\makebox(0,0)[b]{\ Fig.~1}}
\end{picture}
\end{center}

An analogue of Woodall's theorem was given by Manoussakis and Millis in \cite{MM}, and recently considerably strengthened by the authors.

\begin{theorem}[{Adamus \& Adamus, 2012, \cite{AA}}]
\label{thm:AA}
Let $D$ be a bipartite digraph with colour classes $X$ and $Y$ such that $|X|=|Y|=a$, where $a\geq2$.
If $d_D^+(u)+d_D^-(v)\geq a+2$ for every pair of vertices $u$ and $v$ from the opposite colour classes such that $uv\notin A(D)$, then $D$ contains an oriented cycle of length $2a$.
\end{theorem}

In the present paper, we give bipartite analogues of the Ghouila-Houri and Meyniel theorems. These are Theorems~\ref{thm:main-corollary} (below) and~\ref{thm:main}, respectively.
Quite surprisingly, the bounds on degrees are much bigger than one might expect from Theorems~\ref{thm:AM} and~\ref{thm:AA} above.

\begin{theorem}
\label{thm:main-corollary}
Let $D$ be a balanced bipartite digraph with colour classes of cardinalities $a$, where $a\geq2$.
If $\dl(D)\geq(3a+1)/2$, then $D$ contains an oriented cycle of length $2a$.
\end{theorem}

Of course, Theorem~\ref{thm:main-corollary} is an immediate corollary of Theorem~\ref{thm:main}. The bounds in Theorems~\ref{thm:main} and~\ref{thm:main-corollary} are sharp, as can be seen in the following example.

\begin{example}
\label{ex:1}
Let $a$ be a positive even integer, and let $D'$ be a bipartite digraph with colour classes $X$ and $Y$ such that $X$ (resp. $Y$) is a disjoint union of sets $R,S$ (resp. $U,W$) of cardinality $a/2$ each, and $A(D')$ consists of the following arcs:\\
(a) $ry$, for all $r\in R$ and $y\in Y$,\\
(b) $ux$, for all $u\in U$ and $x\in X$, and\\
(c) $sw$ and $ws$, for all $s\in S$ and $w\in W$.\\
Then every vertex of $D'$ is of degree $3a/2$, hence $\dl(D')=3a/2$, but $D'$ contains no hamiltonian cycle.
\end{example}

\begin{remark}
\label{rem:at-least-3a}
Notice that the above $D'$ is not strongly connected. On the other hand, Amar and Manoussakis~\cite{AM} construct a family of digraphs $D(a,k)$ of order $2a$ which are strongly connected, non-hamiltonian and satisfy $\dl(D(a,k))=a+k$, for any $1\leq k<a/2$ (Example~\ref{ex:2}, below). Therefore, even under the strong-connectedness assumption, one cannot get a better bound on $\dl(D)$ in Theorem~\ref{thm:main-corollary} than $3a/2$ (nor a better bound in Theorem~\ref{thm:main} than $3a$).

At the same time, under the assumptions of Theorems~\ref{thm:main} and~\ref{thm:main-corollary}, the strong-connectedness is redundant. In fact, condition $\M$ of Theorem~\ref{thm:main} implies a much stronger property: a bipartite digraph $D$ satisfying condition $\M$ contains a complete matching $M$, and, for every pair of distinct vertices $u,v$, $D$ contains an oriented path from $u$ to $v$ which is compatible with $M$ (cf. Lemma~\ref{lem:3}). 
\end{remark}

\begin{example}
\label{ex:2}
For $a\geq 2$ and $1\leq k<a/2$, let $D(a,k)$ be a bipartite digraph with colour classes $X$ and $Y$ such that $X$ (resp. $Y$) is a disjoint union of sets $R,S$ (resp. $U,W$) with $|R|=|U|=k$, $|S|=|W|=a-k$, and $A(D(a,k))$ consists of the following arcs:\\
(a) $ry$ and $yr$, for all $r\in R$ and $y\in Y$,\\
(b) $ux$ and $xu$, for all $u\in U$ and $x\in X$, and\\
(c) $sw$, for all $s\in S$ and $w\in W$.
\end{example}

Finally, notice that condition $\M$ cannot be weakened to apply only to pairs of vertices from the opposite colour classes ({\`a} la Theorem~\ref{thm:AA}). This follows from the fact that there exist strongly connected non-hamiltonian bipartite tournaments (Example~\ref{ex:3} below). Recall that a \emph{bipartite tournament} is a bipartite digraph $D$ in which, for every pair of vertices $x,y$ from the opposite colour classes, precisely one of the arcs $xy$, $yx$ belongs to $A(D)$.

\begin{example}
\label{ex:3}
For $a\geq 2$ and $1\leq k<a/2$, let $T(a,k)$ be a bipartite digraph with colour classes $X$ and $Y$ such that $X$ (resp. $Y$) is a disjoint union of sets $R,S$ (resp. $U,W$) with $|R|=|U|=k$, $|S|=|W|=a-k$, and $A(T(a,k))$ consists of the following arcs:\\
(a) $ru$, for all $r\in R$ and $u\in U$,\\
(b) $us$, for all $u\in U$ and $s\in S$,\\
(c) $sw$, for all $s\in S$ and $w\in W$, and\\
(d) $wr$, for all $w\in W$ and $r\in R$.\\
Then $T(a,k)$ is strongly connected and vacuously satisfies condition $\M$ for every pair of vertices from the opposite colour classes, but $T(a,k)$ contains no hamiltonian cycle.
\end{example}

\subsection{Notation and terminology}
\label{subsec:not}

A \emph{digraph} $D$ is a pair $(V(D),A(D))$, where $V(D)$ is a finite set (of \emph{vertices}) and $A(D)$ is a set of ordered pairs of distinct elements of $V(D)$, called \emph{arcs} (i.e., $D$ has no loops or multiple arcs). For vertices $u$ and $v$ from $V(D)$, we write $uv\in A(D)$ to say that $A(D)$ contains the ordered pair $(u,v)$. The number of vertices $|V(D)|$ is the \emph{order} of $D$ (also denoted by $|D|$). The \emph{size} of $D$, denoted $\|D\|$, is defined as $|A(D)|$.

For vertex sets $S,T\subset V(D)$, we denote by $N_S^+(T)$ the set of vertices in $S$ \emph{dominated} by the vertices of $T$; i.e.,
\[
N_S^+(T)=\{u\in S: vu\in A(D)\text{\ for\ some\ }v\in T\}\,.
\]
Similarly, $N_S^-(T)$ denotes the set of vertices of $S$ \emph{dominating} the vertices of $T$; i.e,
\[
N_S^-(T)=\{u\in S: uv\in A(D)\text{\ for\ some\ }v\in T\}\,.
\]
If $T=\{v\}$ is a single vertex, the cardinality of $N_S^+(v)$ (resp. $N_S^-(v)$), denoted by $d_S^+(v)$ (resp. $d_S^-(v)$) is called the
\emph{outdegree} (resp. \emph{indegree}) of $v$ relative to $S$. The \emph{degree} of $v$ (relative to $S$) is $d_S(v)=d_S^+(v)+d_S^-(v)$.
To simplify notation, for a sub-digraph $D'$ of $D$, we will often write $d^+_{D'}(u)$ (resp. $d^-_{D'}(u)$, or $d_{D'}(u)$) instead of $d^+_{V(D')}(u)$ (resp. $d^-_{V(D')}(u)$, or $d_{V(D')}(u)$). Also, we will write $N^+(T)$ (resp. $N^-(T)$) for $N_{V(D)}^+(T)$ (resp. $N_{V(D)}^-(T)$). Further, by $\dlp(D)$ and $\dlm(D)$ we will denote respectively the least outdegree and the least indegree of $D$; i.e., $\dlp(D)=\min\{d_D^+(v):v\in V(D)\}$ and $\dlm(D)=\min\{d_D^-(v):v\in V(D)\}$. The minimal degree of $D$ will be denoted by $\dl(D)$.

A digraph \emph{induced} in $D$ by a vertex subset $S\subset V(D)$ is denoted by $D[S]$, and $D-S$ denotes a digraph obtained from $D$ by removing the vertices of $S$ and their incident arcs (that is, $D-S=D[V(D)\setminus S]$).

An oriented cycle (resp. oriented path) on vertices $v_1,\dots,v_m$ in $D$ is denoted by $[v_1,\ldots,v_m]$ (resp. $(v_1,\ldots,v_m)$). We will refer to them as simply \emph{cycles} and \emph{paths} (skipping the term ``oriented''), since their non-oriented counterparts are not considered in this article at all.

A cycle passing through all the vertices of $D$ is called \emph{hamiltonian}. A digraph containing a hamiltonian cycle is called a \emph{hamiltonian digraph}.

A digraph $D$ is \emph{bipartite} when $V(D)$ is a disjoint union of sets $X$ and $Y$ (the \emph{colour classes}) such that $A(D)\cap(X\times X)=\varnothing$ and $A(D)\cap (Y\times Y)=\varnothing$. It is called \emph{balanced} if $|X|=|Y|$. A \emph{matching from $X$ to $Y$} is an independent set of arcs with origin in $X$ and terminus in $Y$. If $D$ is balanced, one says that such a matching is \emph{complete} if it consists of precisely $|X|$ arcs. A path or cycle is said to be \emph{compatible} with a matching $M$ from $X$ to $Y$ (or, \emph{$M$-compatible}, for short) if its arcs are alternately in $M$ and in $A(D)\setminus M$.

For a complete matching $M$ from $X$ to $Y$ and a vertex $x'\in X$, we will denote by $M(x')$ the unique vertex $y'\in Y$ such that $x'y'\in M$. Similarly, for $y'\in Y$, we will denote by $M^{-1}(y')$ the unique vertex $x'\in X$ for which $x'y'\in M$. Finally, for a subset $S\subset Y$, we will denote by $M^{-1}(S)$ the set $\{M^{-1}(y):y\in S\}$.

\subsection{Plan of the proof} We prove Theorem~\ref{thm:main} in Section~\ref{sec:main-proof}, after establishing its technical components in a series of lemmas below. We proceed by contradiction. Despite its discouraging length, the main idea of the proof is fairly simple: First, we show that, under condition $\M$, our bipartite digraph $D$ splits into a sequence of cycles $C_1,\dots,C_k$ with pairwise disjoint vertex sets, such that each consecutive cycle contains at least half the vertices remaining after removing its predecessing cycles, and is of maximal possible length. The key component here is our observation that condition $\M$ is essentially hereditary in this decomposition (cf. Lemma~\ref{lem:5}). More precisely, if $D$ satisfies condition $\M$, then, for every $j\in\{1,\dots,k-1\}$, $D-(V(C_1)\cup\dots\cup V(C_j))$ satisfies the so-called condition $\A$ (see Def.~\ref{def:6a'+2}), which is but condition $\M$ applied to $4$-tuples rather than pairs of vertices. This observation allows us to work recursively in the digraphs $D-(V(C_1)\cup\dots\cup V(C_j))$, $j\geq1$. Next, we show that, for some $j$, $D$ contains an oriented path $P$ which originates and terminates in $C_j$ and passes through all the cycles ``to the right'' of $C_j$ (i.e., $C_{j+1},\dots$). Finally, we prove that $P$ is, in fact, so long that glueing it into $C_j$ produces a cycle strictly longer than $C_j$, which contradicts its maximality.

\medskip

\section{Lemmata}
\label{sec:lemmata}

\begin{lemma}
\label{lem:1}
Let $D$ be a balanced bipartite digraph with colour classes of cardinalities $a$, where $a\geq2$.
If $D$ satisfies condition $\M$, then for every set of vertices $S$ contained in one of the colour classes of $D$ and with cardinality $|S|\leq(a+1)/2$, we have $|N^+(S)|\geq|S|$.
\end{lemma}

\begin{proof}
First observe that condition $\M$ implies $d_D^+(v)>0$ for every $v\in V(D)$. For if $d_D^+(v)=0$, then, for any $v'\in V(D)$ from the same colour class, one has $d_D(v)+d_D(v')=d_D^+(v)+d_D^-(v)+d_D(v')\leq0+a+2a<3a+1$, which contradicts condition $\M$.

Let then $S$ be a set of vertices of $D$ contained in one of the colour classes and such that $|S|\leq(a+1)/2$. If $|S|\leq1$, then, by the above observation, there is nothing to show. One can thus assume that $S$ contains two distinct vertices, say $v_1$ and $v_2$. Suppose that $|N^+(S)|<|S|$. Then
\begin{multline}
\notag
d_D(v_1)+d_D(v_2)=(d_D^-(v_1)+d_D^-(v_2))+(d_D^+(v_1)+d_D^+(v_2))\leq 2a+2|N^+(S)|\\
<2a+2|S|\leq2a+(a+1)=3a+1\,,
\end{multline}
which contradicts condition $\M$ again.
\end{proof}

\begin{lemma}
\label{lem:2}
Let $D$ be a balanced bipartite digraph with colour classes $X$ and $Y$ of cardinalities $a$, where $a\geq2$.
If $D$ satisfies condition $\M$, then $D$ contains a complete matching from $X$ to $Y$ or a complete matching from $Y$ to $X$.
\end{lemma}

\begin{proof}
For a proof by contradiction, suppose that $D$ contains no complete matching from $X$ to $Y$ nor from $Y$ to $X$.
Then, by Hall's theorem (see, e.g., \cite{B}), there exist sets $S\subset X$ and $T\subset Y$ such that $|N^+(S)|<|S|$ and $|N^+(T)|<|T|$. Define
\begin{align}
\notag
s_X\ &=\ \min\{j\in\N:\text{\ there\ is\ }S\subset X\text{\ such\ that\ }|S|=j,\ |N^+(S)|<|S|\}\quad\text{and}\\
\notag
s_Y\ &=\ \min\{j\in\N:\text{\ there\ is\ }T\subset X\text{\ such\ that\ }|T|=j,\ |N^+(T)|<|T|\}\,.
\end{align}
Without loss of generality, we can assume that $s_Y\leq s_X$. By Lemma~\ref{lem:1}, both $s_X$ and $s_Y$ are strictly greater than $(a+1)/2$.

Pick subsets $S_0\subset X$ and $T_0\subset Y$ such that $|S_0|=s_X$, $|T_0|=s_Y$, $|N^+(S_0)|<|S_0|$, and $|N^+(T_0)|<|T_0|$.
We have $S_0\cap(X\setminus N^+(T_0))\neq\varnothing$, for else $S_0\subset N^+(T_0)$, hence $s_X=|S_0|\leq|N^+(T_0)|<|T_0|=s_Y$, contrary to our assumption.

Now, for every $x\in S_0\setminus N^+(T_0)$, we have $d_D^+(x)\leq|N^+(S_0)|<s_X$ and $d_D^-(x)\leq a-|T_0|=a-s_Y$, hence
\begin{equation}
\label{eq:1elt}
d_D(x)<a+s_X-s_Y\,.
\end{equation}
Therefore, if $S_0\setminus N^+(T_0)$ contains at least two elements, $x_1$ and $x_2$, say, then
\begin{multline}
\notag
d_D(x_1)+d_D(x_2)<2(a+s_X-s_Y)=2a+2s_X-2s_Y\\
<2a+2s_X-(a+1)=2s_X+a-1\,.
\end{multline}
On the other hand, by condition $\M$, $d_D(x_1)+d_D(x_2)\geq3a+1$. It follows that $s_X>a+1$, which is absurd.

It thus remains to consider the case that $|S_0\setminus N^+(T_0)|=1$. One then has $|S_0|=|S_0\cap N^+(T_0)|+1$, hence $s_X=|S_0|\leq|N^+(T_0)|+1\leq(s_Y-1)+1=s_Y$, and so $s_X=s_Y$. Let $x_0$ denote the sole vertex of $S_0\setminus N^+(T_0)$. By \eqref{eq:1elt}, we now have $d_D(x_0)<a$. Therefore, for any vertex $x\in X\setminus\{x_0\}$,
\[
d_D(x_0)+d_D(x)<a+2a<3a+1\,,
\]
which contradicts condition $\M$.
\end{proof}

\begin{lemma}
\label{lem:3}
Let $D$ be a balanced bipartite digraph with colour classes $X$ and $Y$ of cardinalities $a$, where $a\geq2$, which satisfies condition $\M$. Suppose that $D$ contains a complete matching $M$ from $X$ to $Y$, and $D$ contains no oriented cycle of lenght $2a$. Then, for every pair of distinct vertices $u,v\in V(D)$, $D$ contains an $M$-compatible path from $u$ to $v$.
\end{lemma}

\begin{remark}
\label{rem:non-hamiltonian}
Under the hypotheses of Lemma~\ref{lem:3}, notice that $d_D^+(v)\geq2$ and $d_D^-(v)\geq2$ for all $v\in V(D)$. Indeed, for if $d_D^+(v')<2$ for some $v'\in V(D)$, then $d_D(v')\leq a+1$, hence, by condition $\M$, $d_D(v)\geq2a$ for all $v\neq v'$ from the same colour class. Since every degree is bounded above by $2a$, we would actually have $d_D(v)=2a$ for all $v\neq v'$ from the colour class of $v'$, as well as $d_D^+(v')=1$ and $d_D^-(v')=a$. It is readily seen that then $D$ would contain a hamiltonian cycle. The argument for $d_D^-(v)$ is analogous.
\end{remark}

\subsubsection*{Proof of Lemma~\ref{lem:3}}
First, we claim that it suffices to show that $D$ contains an $M$-compatible path from $y$ to $x$ for every pair of vertices such that $y\in Y$ and $x\in X$. Indeed, to find an $M$-compatible path in $D$ from $x'\in X$ to $x''\in X$, it suffices to find an $M$-compatible path from $M(x')$ to $x''$. Likewise, to find an $M$-compatible path from $y'\in Y$ to $y''\in Y$, it suffices to find an $M$-compatible path from $y'$ to $M^{-1}(y'')$. Finally, to find an $M$-compatible path from $x'\in X$ to $y''\in Y$, it suffices to find an $M$-compatible path from $M(x')$ to $M^{-1}(y'')$ (unless $x'y''$ already is in $M$).\smallskip

For a proof by contradiction, suppose that $y'\in Y$ and $x''\in X$ are such that $D$ contains no path from $y'$ to $x''$ compatible with $M$. By Remark~\ref{rem:non-hamiltonian}, we have $d_D^+(y')\geq2$ and $d_D^-(x'')\geq2$.
Denote by $S$ the set of those vertices in $Y\setminus\{y'\}$ to which one can get from $y'$ along an $M$-compatible path of positive length. We have $|S|\geq1$, since $d_D^+(y')\geq2$. Moreover, by hypothesis, $x''\in X\setminus N^+(S)$, and so
\begin{equation}
\label{eq:x''}
d_D^-(x'')\leq a-|S|\,.
\end{equation}
Let $x'$ denote $M^{-1}(y')$ and let $y''$ denote $M(x'')$. (It may be that $x'=x''$ and $y'=y''$.)

Choose a vertex $y'''\in Y\setminus\{y''\}$ such that $y'''x''\in A(D)$. Such a vertex exists, since $d_D^-(x'')\geq2$. Note that $y'''\neq y'$, for otherwise $D$ would contain an $M$-compatible path from $y'$ to $x''$ (namely, the arc $y'x''$ itself). For the same reason, the vertex $x''':=M^{-1}(y''')$ belongs to $X\setminus N^+(S)$ and is not dominated by $y'$. Consequently,
\begin{equation}
\label{eq:x'''}
d_D^-(x''')\leq a-(|S|+1)\,.
\end{equation}
Now, condition $\M$ together with \eqref{eq:x''} and \eqref{eq:x'''} imply that
\begin{multline}
\notag
3a+1\leq d_D(x'')+d_D(x''')=(d_D^+(x'')+d_D^+(x'''))+(d_D^-(x'')+d_D^-(x'''))\\
\leq 2a+(a-|S|)+(a-|S|-1)\,,
\end{multline}
hence $|S|\leq(a-2)/2$.

On the other hand, by definition of $S$, $N^+(y')\subset\{x'\}\cup M^{-1}(S)$, and $N^+(y)\subset\{x'\}\cup M^{-1}(S)$ for all $y\in S$. Hence, for any $y\in S$, we have
\begin{multline}
\notag
3a+1\leq d_D(y')+d_D(y)=(d_D^-(y')+d_D^-(y))+(d_D^+(y')+d_D^+(y))\\
\leq 2a+2(|M^{-1}(S)|+1)=2a+2|S|+2\,,
\end{multline}
and so $|S|\geq(a-1)/2$; a contradiction.
\qed
\medskip

\begin{definition}
\label{def:6a'+2}
Consider a balanced bipartite digraph $D'$ with colour classes $X'$ and $Y'$ of cardinalities $a'$, where $a'\geq2$. Suppose that $D'$ contains a complete matching $M'$ from $X'$ to $Y'$. We will say that $D'$ satisfies \emph{condition $\A$} when
\[
d_{D'}(x')+d_{D'}(y')+d_{D'}(x'')+d_{D'}(y'')\geq 6a'+2
\]
for all pairwise distinct $x',x''\in X'$ and $y',y''\in Y'$ such that $D'$ contains $M'$-compatible paths from $x'$ to $y'$ and from $x''$ to $y''$.
\end{definition}

Notice that condition $\A$ follows from, but is strictly weaker than condition $\M$.

\begin{lemma}
\label{lem:4}
Let $D'$ be a balanced bipartite digraph with colour classes $X'$ and $Y'$ of cardinalities $a'$, where $a'\geq2$, and let $M'$ be a complete matching from $X'$ to $Y'$ in $D'$. Suppose that $D'$ satisfies condition $\A$.
If $a'\geq3$, then $D'$ contains an $M'$-compatible cycle of length at least $a'$. If $a'=2$, then $D'$ contains a cycle of length $4$ compatible with some matching from $X'$ to $Y'$.
\end{lemma}

\begin{proof}
Suppose first that $a'=2$. Then, we can write $X'=\{x',x''\}$ and $Y'=\{y',y''\}$, where $M'$ consists of $x'y'$ and $x''y''$. By assumption,
\[
\|D'\|=\frac{1}{2}\left( d_{D'}(x')+d_{D'}(y')+d_{D'}(x'')+d_{D'}(y'')\right) \geq\frac{6a'+2}{2}=7\,,
\]
and so $D'$ is obtained from a complete bipartite digraph of order $4$ by deleting at most one arc. Clearly, such a digraph contains a hamiltonian cycle, and the cycle determines a complete matching from $X'$ to $Y'$ with which it is compatible.\medskip

Suppose then that $a'\geq3$. Note that $D'$ contains a vertex $x'\in X'$ such that $d_{D'}^-(x')\geq2$ or a vertex $y'\in Y'$ such that $d_{D'}^+(y')\geq2$. Indeed, for if $d_{D'}^-(x)\leq1$ for all $x\in X'$ and $d_{D'}^+(y)\leq1$ for all $y\in Y'$, then choosing $x',x''\in X'$ and $y',y''\in Y'$ such that $x'y',x''y''\in M'$, we would get
\[
6a'+2\leq d_{D'}(x')+d_{D'}(y')+d_{D'}(x'')+d_{D'}(y'')\leq 4(a'+1)\,,
\]
hence $a'\leq1$; a contradiction.

Consequently, $D'$ contains an $M'$-compatible path of order at least $4$. Let $P$ denote an $M'$-compatible path in $D'$ of maximal length (among all such paths). By maximality, we can assume that the initial vertex of $P$ belongs to $X'$ and its terminal vertex belongs to $Y'$ (see \cite[Rem.\,2.2]{AA} for an overkill argument). Therefore, we can write $P=(x_1,y_1,\dots,x_s,y_s)$ for some $x_1,\dots,x_s\in X'$ and $y_1,\dots,y_s\in Y'$, where $s\geq2$. Also, by maximality of $P$, we have
\begin{equation}
\label{eq:lem41}
N^+(y_s)\subset V(P)\cap X'\quad\text{and}\quad N^-(x_1)\subset V(P)\cap Y'\,.
\end{equation}
As far as the vertex $y_{s-1}$ is concerned, there are two possibilities. Either

(a) $N^+(y_{s-1})\subset V(P)\cap X'$, or else

(b) $N^+(y_{s-1})\not\subset V(P)$.\\
In case (b), there exist $x_{s+1}\in X'\setminus V(P)$ and $y_{s+1}\in Y'\setminus V(P)$ such that $y_{s-1}x_{s+1}\in A(D')$ and $x_{s+1}y_{s+1}\in M'$. The new path $(x_1,y_1,\dots,x_{s-1},y_{s-1},x_{s+1},y_{s+1})$ is also $M'$-compatible of maximal length, and hence
\begin{equation}
\label{eq:lem43}
N^+(y_{s+1})\subset\{x_1,\dots,x_{s-1},x_{s+1}\}\,.
\end{equation}
Similarly, for the vertex $x_2$, we have either

(c) $N^-(x_2)\subset V(P)\cap Y'$, or else

(d) $N^-(x_2)\not\subset V(P)$.\\
In case (d),  there exist $x_0\in X'\setminus V(P)$ and $y_0\in Y'\setminus V(P)$ such that $x_0y_0\in M'$ and $y_0x_2\in A(D')$. The new path $(x_0,y_0,x_2,y_2,\dots,y_s,x_s)$ is also $M'$-compatible of maximal length, and hence
\begin{equation}
\label{eq:lem45}
N^-(x_0)\subset\{y_0,y_2,\dots,y_s\}\,.
\end{equation}

The rest of the proof proceeds in four cases, according to the combinations of the above conditions (a) -- (d). We claim that $D'$ contains an $M'$-compatible cycle of length at least $a'$. Suppose otherwise.

\subsubsection*{Case 1.} Suppose first that (a) and (c) hold.\\
We will apply condition $\A$ to the endpoints of the $M'$-compatible paths $(x_1,y_1,\dots,$ $x_{s-1},y_{s-1})$ and $(x_2,y_2,\dots,x_s,y_s)$. By condition $\A$, together with properties (a), (c) and \eqref{eq:lem41}, we get
\begin{multline}
\notag
6a'+2\leq d_{D'}(x_1)+d_{D'}(y_{s-1})+d_{D'}(x_2)+d_{D'}(y_s)\\
=(d_{D'}^-(x_1)+d_{D'}^+(y_{s-1})+d_{D'}^-(x_2)+d_{D'}^+(y_s))+(d_{D'}^+(x_1)+d_{D'}^-(y_{s-1})+d_{D'}^+(x_2)+d_{D'}^-(y_s))\\
\leq 4s+4a'\,,
\end{multline}
hence $s\geq(a'+1)/2$.

Now, $y_sx_1\notin A(D')$, for else $D'$ would contain an $M'$-compatible cycle $[x_1,y_1,\dots,$ $x_s,y_s]$ of length $2s\geq a'+1$. Therefore, we have
\[
N^+(y_s)\subset\{x_2,\dots,x_s\}\quad\text{and}\quad N^-(x_1)\subset\{y_1,\dots,y_{s-1}\}\,,
\]
and hence, by condition $\A$ again,
\[
6a'+2\leq (d_{D'}^-(x_1)+d_{D'}^+(y_{s-1})+d_{D'}^-(x_2)+d_{D'}^+(y_s))+4a'\leq (4s-2)+4a'\,,
\]
so that $s\geq(a'+2)/2$. In particular, $s\geq3$ (because $a'\geq3$), and thus $y_{s-1}\neq y_1$ and $x_2\neq x_s$. Moreover, $y_sx_2\notin A(D')$, for else $D'$ would contain an $M'$-compatible cycle $[x_2,y_2,\dots,x_s,y_s]$ of length $2(s-1)\geq a'$. Therefore
\begin{equation}
\label{eq:lem46}
N^+(y_s)\subset\{x_3,\dots,x_s\}\quad\text{and}\quad N^-(x_2)\subset\{y_1,\dots,y_{s-1}\}\,.
\end{equation}
Similarly, $y_{s-1}x_1\notin A(D')$, for else $D'$ would contain an $M'$-compatible cycle $[x_1,y_1,\dots,$ $x_{s-1},y_{s-1}]$ of length $2(s-1)\geq a'$. Therefore
\begin{equation}
\label{eq:lem47}
N^+(y_{s-1})\subset\{x_2,\dots,x_s\}\quad\text{and}\quad N^-(x_1)\subset\{y_1,\dots,y_{s-2}\}\,.
\end{equation}
Hence, condition $\A$ together with \eqref{eq:lem46} and \eqref{eq:lem47} imply that
\[
6a'+2\leq (d_{D'}^-(x_1)+d_{D'}^+(y_{s-1})+d_{D'}^-(x_2)+d_{D'}^+(y_s))+4a'\leq (4s-6)+4a'\,,
\]
so that $s\geq(a'+4)/2$. In particular, $s\geq4$ (because $a'\geq3$), and thus $y_{s-2}\neq y_1$ and $x_3\neq x_s$. Moreover, $y_sx_3\notin A(D')$, for else $D'$ would contain an $M'$-compatible cycle $[x_3,y_3,\dots,x_s,y_s]$ of length $2(s-2)\geq a'$. Therefore
\begin{equation}
\label{eq:lem48}
N^+(y_s)\subset\{x_4,\dots,x_s\}\,.
\end{equation}
Similarly, $y_{s-1}x_2\notin A(D')$, for else $D'$ would contain an $M'$-compatible cycle $[x_2,y_2,\dots,$ $x_{s-1},y_{s-1}]$ of length $2(s-2)\geq a'$. Therefore
\begin{equation}
\label{eq:lem49}
N^+(y_{s-1})\subset\{x_3,\dots,x_s\}\quad\text{and}\quad N^-(x_2)\subset\{y_1,\dots,y_{s-2}\}\,.
\end{equation}
Finally, $y_{s-2}x_1\notin A(D')$, for else $D'$ would contain an $M'$-compatible cycle $[x_1,y_1,$ $\dots,x_{s-2},y_{s-2}]$ of length $2(s-2)\geq a'$. Therefore
\begin{equation}
\label{eq:lem410}
N^-(x_1)\subset\{y_1,\dots,y_{s-3}\}\,.
\end{equation}
Hence, condition $\A$ together with \eqref{eq:lem48}, \eqref{eq:lem49} and \eqref{eq:lem410} imply that
\[
6a'+2\leq (d_{D'}^-(x_1)+d_{D'}^+(y_{s-1})+d_{D'}^-(x_2)+d_{D'}^+(y_s))+4a'\leq (4s-10)+4a'\,,
\]
so that $s\geq(a'+6)/2$. And so on...\medskip

One readily sees that, by continuing the above procedure, we eventually obtain $s\geq a'$; i.e., $V(P)=V(D')$. Then, by condition $\A$,
\[
6a'+2\leq (d_{V(P)}^-(x_1)+d_{V(P)}^-(x_2)+d_{V(P)}^+(y_{s-1})+d_{V(P)}^+(y_s))+4a'\,,
\]
hence $d_{V(P)}^-(x_1)+d_{V(P)}^-(x_2)\geq a'+1$ or $d_{V(P)}^+(y_{s-1})+d_{V(P)}^+(y_s)\geq a'+1$. Without loss of generality, suppose that the latter inequality holds. Then, either $d_{V(P)}^+(y_s)\geq(a'+1)/2$ or else $d_{V(P)}^+(y_{s-1})\geq(a'+1)-a'/2=(a'+2)/2$.

Now, if $d_{V(P)}^+(y_s)\geq(a'+1)/2$, then there exists $j\leq(a'+1)/2$ such that $y_sx_j\in A(D')$. Then $D'$ contains an $M'$-compatible cycle $[x_j,y_j,\dots,x_s,y_s]$ of length at least $2(a'-(a'-1)/2)=a'+1$; a contradiction. If, in turn, $d_{V(P)}^+(y_{s-1})\geq(a'+2)/2$, then there exists $j\leq a'/2$ such that $y_{s-1}x_j\in A(D')$. Then $D'$ contains an $M'$-compatible cycle $[x_j,y_j,\dots,x_{s-1},y_{s-1}]$ of length at least $2((a'-1)-(a'/2-1))=a'$. The contradiction completes the proof of Case 1.

\subsubsection*{Case 2.} Suppose now that (a) and (d) hold.\\
We will apply condition $\A$ to the endpoints of the $M'$-compatible paths $(x_1,y_1,\dots,$ $x_{s-1},y_{s-1})$ and $(x_0,y_0,x_2,y_2,\dots,x_s,y_s)$. By condition $\A$ together with \eqref{eq:lem41}, (a) and \eqref{eq:lem45},
\[
6a'+2\leq (d_{D'}^-(x_1)+d_{D'}^+(y_{s-1})+d_{D'}^-(x_0)+d_{D'}^+(y_s))+4a'\leq 4s+4a'\,,
\]
hence $s\geq(a'+1)/2$.

Now, $y_sx_1\notin A(D')$, for else $D'$ would contain an $M'$-compatible cycle $[x_1,y_1,$ $\dots,x_s,y_s]$ of length $2s\geq a'+1$. Therefore, we have
\[
N^+(y_s)\subset\{x_2,\dots,x_s\}\quad\text{and}\quad N^-(x_1)\subset\{y_1,\dots,y_{s-1}\}\,.
\]
Also, by maximality of $P$, $y_sx_0\notin A(D')$. Therefore, we have
\[
N^-(x_0)\subset\{y_0,y_2,\dots,y_{s-1}\}\,,
\]
and hence, by condition $\A$ again,
\[
6a'+2\leq (d_{D'}^-(x_1)+d_{D'}^+(y_{s-1})+d_{D'}^-(x_0)+d_{D'}^+(y_s))+4a'\leq (4s-3)+4a'\,,
\]
so that $s\geq(a'+\frac{5}{2})/2$. In particular, $s\geq3$ (because $a'\geq3$), and thus $y_{s-1}\neq y_1$. By (d), also $y_{s-1}\neq y_0$. Moreover, $y_sx_2\notin A(D')$, for else $D'$ would contain an $M'$-compatible cycle $[x_2,y_2,\dots,x_s,y_s]$ of length $2(s-1)\geq a'+1/2$. Therefore
\begin{equation}
\label{eq:lem411}
N^+(y_s)\subset\{x_3,\dots,x_s\}\,.
\end{equation}
Similarly, $y_{s-1}x_1\notin A(D')$, for else $D'$ would contain an $M'$-compatible cycle $[x_1,y_1,\dots,$ $x_{s-1},y_{s-1}]$ of length $2(s-1)\geq a'+1/2$. Therefore
\begin{equation}
\label{eq:lem412}
N^+(y_{s-1})\subset\{x_2,\dots,x_s\}\quad\text{and}\quad N^-(x_1)\subset\{y_1,\dots,y_{s-2}\}\,.
\end{equation}
By (a), also $y_{s-1}x_0\notin A(D')$. Therefore
\begin{equation}
\label{eq:lem413}
N^-(x_0)\subset\{y_0,y_2,\dots,y_{s-2}\}\,.
\end{equation}
Hence, condition $\A$ together with \eqref{eq:lem411}, \eqref{eq:lem412} and \eqref{eq:lem413} imply that
\[
6a'+2\leq (d_{D'}^-(x_1)+d_{D'}^+(y_{s-1})+d_{D'}^-(x_0)+d_{D'}^+(y_s))+4a'\leq (4s-7)+4a'\,,
\]
so that $s\geq(a'+\frac{9}{2})/2$.

As in Case 1, after finitely many steps, the above procedure terminates with $s\geq a'$; i.e., $V(P)=V(D')$. This, however, is impossible, because, by condition (d), $P$ does not contain vertices $x_0$ and $y_0$. The contradiction completes the proof of Case 2.

\subsubsection*{Case 3.} Suppose now that (b) and (c) hold.\\
This configuration is analogous to the one in Case 2. The reader may easily adapt the proof of Case 2 to the $M'$-compatible paths $(x_2,y_2,\dots,x_s,y_s)$ and $(x_1,y_1,\dots,x_{s-1},y_{s-1},x_{s+1},y_{s+1})$.

\subsubsection*{Case 4.} Finally, suppose that (b) and (d) hold.\\
First, we want to rule out the possibility that $x_0=x_{s+1}$ (hence also $y_0=y_{s+1}$). If that were the case, then, to simplify notation, set $x':=x_0=x_{s+1}$ and $y':=y_0=y_{s+1}$. Observe that $y_{s+1}x_1\notin A(D')$ and $y_sx_0\notin A(D')$, by maximality of $P$ (see \eqref{eq:lem41}). Hence, by \eqref{eq:lem43} and \eqref{eq:lem45}, $d_{D'}^+(y')\leq s-1$ and $d_{D'}^-(x')\leq s-1$. Therefore, by applying condition$\A$ to the endpoints of the $M'$-compatible paths $(x_1,y_1,\dots,x_s,y_s)$ and $x'y'$, we get
\[
6a'+2\leq (d_{D'}^-(x_1)+d_{D'}^+(y_s)+d_{D'}^-(x')+d_{D'}^+(y'))+4a'\leq (4s-2)+4a'\,,
\]
hence $s\geq(a'+2)/2$. On the other hand, by (b), we have $y_{s-1}x'=y_{s-1}x_{s+1}\in A(D')$, and so $D'$ contains an $M'$-compatible cycle $[x_0,y_0,x_2,y_2,\dots,x_{s-1},y_{s-1}]$ of length $2(s-1)\geq a'$; a contradiction.

We thus have $x_0\neq x_{s+1}$, and hence $y_0\neq y_{s+1}$. Consequently, the $M'$-compatible paths $(x_1,y_1,\dots,x_{s-1},y_{s-1},x_{s+1},y_{s+1})$ and $(x_0,y_0,x_2,y_2,\dots,x_s,y_s)$ have pairwise distinct initial and terminal points. One can, once more, easily adapt the argument of Case 2 to these paths. This completes the proof of the lemma.
\end{proof}

\begin{lemma}
\label{lem:5}
Let $D$ be a balanced bipartite digraph with colour classes $X$ and $Y$ of cardinalities $a$, where $a\geq2$, and let $M$ be a complete matching from $X$ to $Y$ in $D$. Suppose that $D$ contains $M$-compatible cycles $C_1,\dots,C_l$ (of length at least $4$ each) such that $C_1$ is of maximal length among all $M$-compatible cycles in $D$ and, for every $1\leq j<l$, $C_{j+1}$ is of maximal length among all $M$-compatible cycles in $D-(V(C_1)\cup\dots\cup V(C_j))$. Set $D'=D-(V(C_1)\cup\dots\cup V(C_l))$ and $a'=|D'|/2$.
If $D$ satisfies condition $\M$ and $a'\geq2$, then $D'$ satisfies condition $\A$, that is,
\[
d_{D'}(x')+d_{D'}(y')+d_{D'}(x'')+d_{D'}(y'')\geq 6a'+2
\]
for all pairwise distinct $x',x''\in V(D')\cap X$ and $y',y''\in V(D')\cap Y$ such that $D'$ contains $M$-compatible paths from $x'$ to $y'$ and from $x''$ to $y''$.
\end{lemma}

\begin{proof}
Choose pairwise distinct $x',x''\in V(D')\cap X$ and $y',y''\in V(D')\cap Y$ such that $D'$ contains $M$-compatible paths from $x'$ to $y'$ and from $x''$ to $y''$. Note that, for every $1\leq j\leq l$,
\[
d_{V(C_j)}^-(x')+d_{V(C_j)}^+(y')\leq|C_j|/2\quad\text{and}\quad d_{V(C_j)}^-(x'')+d_{V(C_j)}^+(y'')\leq|C_j|/2\,.
\]
Indeed, for if, for instance, $d_{V(C_{j_0})}^-(x')+d_{V(C_{j_0})}^+(y')>|C_{j_0}|/2$ for some $j_0\in\{1,\dots,l\}$, then $C_{j_0}$ contains an arc $y^*x^*$ such that $y^*x',y'x^*\in A(D)$. Replacing $y^*x^*$ in $C_{j_0}$ with the path $(y^*,x',\dots,y',x^*)$ gives an $M$-compatible cycle in $D-(V(C_1)\cup\dots\cup V(C_{j_0-1}))$ of length strictly greater than $|C_{j_0}|$, which contradicts the choice of $C_{j_0}$.\\
Now, condition $\M$ implies
\begin{multline}
\notag
2(3a+1)\leq(d_D(x')+d_D(x''))+(d_D(y')+d_D(y''))=\\
\sum_{j=1}^l\left( (d_{V(C_j)}^-(x')+d_{V(C_j)}^+(y'))+(d_{V(C_j)}^-(x'')+d_{V(C_j)}^+(y''))\right)\\
+\sum_{j=1}^l\left( d_{V(C_j)}^+(x')+d_{V(C_j)}^-(y')+d_{V(C_j)}^+(x'')+d_{V(C_j)}^-(y'')\right)\\
+\left( d_{D'}(x')+d_{D'}(y')+d_{D'}(x'')+d_{D'}(y'')\right)\\
\leq \sum_{j=1}^l2\frac{|C_j|}{2}+\sum_{j=1}^l4\frac{|C_j|}{2}+\left(d_{D'}(x')+d_{D'}(y')+d_{D'}(x'')+d_{D'}(y'')\right)\,,
\end{multline}
hence
\[
d_{D'}(x')+d_{D'}(y')+d_{D'}(x'')+d_{D'}(y'')\geq (6a+2)-6\sum_{j=1}^l\frac{|C_j|}{2}=6a'+2\,,
\]
as required.
\end{proof}

\medskip

\section{Proof of Theorem~\ref{thm:main}}
\label{sec:main-proof}

For a proof by contradiction, suppose that $D$ is a balanced bipartite digraph with colour classes $X$ and $Y$ of cardinalities $a\geq2$, which satisfies condition $\M$ and contains no cycle of length $2a$. By Lemma~\ref{lem:2}, $D$ contains a complete matching from $X$ to $Y$ or from $Y$ to $X$. For the rest of the proof, assume, without loss of generality, that there exists a complete matching from $X$ to $Y$ in $D$.

\subsection{Decomposition into cycles}
\label{subsec:cycles}
First, we shall show that $D$ contains a complete matching $M$ from $X$ to $Y$ and $M$-compatible cycles $C_1,\dots,C_k$ (of length at least $4$ each), for some $k\geq1$, all such that:
\begin{itemize}
\item[(i)] $V(D)=V(C_1)\cup\dots\cup V(C_k)\cup V^r$ is a disjoint union, where $|V^r|=2$ or $V^r=\varnothing$.
\item[(ii)] $C_1$ is of maximal length among all cycles compatible with some complete matching from $X$ to $Y$, and, for every $j=1,\dots,k-1$, $C_{j+1}$ is of maximal length among all cycles compatible with some complete matching from $X\cap V(D-(V(C_1)\cup\dots\cup V(C_j)))$ to $Y\cap V(D-(V(C_1)\cup\dots\cup V(C_j)))$.
\item[(iii)] $|C_1|\geq a$ and, for every $j=1,\dots,k-1$, $C_{j+1}$ passes through at least half the vertices of $D-(V(C_1)\cup\dots\cup V(C_j))$.
\end{itemize}

We will construct $M$ and the cycles $C_1,\dots,C_k$ recursively, by an alternate use of Lemmas~\ref{lem:4} and~\ref{lem:5}:
By assumption, $D$ satisfies condition $\M$, hence also condition $\A$. We can thus apply Lemma~\ref{lem:4} to $D$. By Lemma~\ref{lem:4}, if $a=2$, then $D$ contains a hamiltonian cycle, contrary to our hypothesis. Thus $a\geq3$, and hence, by Lemma~\ref{lem:4} again, there is a cycle in $D$, of length at least $a$, compatible with a complete matching from $X$ to $Y$.

Let $C_1$ be a cycle in $D$ of maximal lenght among all cycles compatible with some complete matching from $X$ to $Y$, and let $M_1$ be a complete matching from $X$ to $Y$ with which $C_1$ is compatible. By assumption, $a-|C_1|/2\geq1$. If, in fact, $a-|C_1|/2=1$, then setting $M=M_1$ we are done. If $a-|C_1|/2\geq2$, then we set $D'=D-V(C_1)$ and apply to it Lemma~\ref{lem:5}, to get that $D'$ satisfies condition $\A$. We can thus apply Lemma~\ref{lem:4} to $D'$. Set $a'=|D'|/2$. If $a'=2$, then Lemma~\ref{lem:4} implies that $D'$ contains a cycle $C_2$ of length $4$, which defines a complete matching $M_2$ from $X\cap V(D')$ to $Y\cap V(D')$. Setting $M$ to coincide with $M_1$ on $V(C_1)$ and with $M_2$ on $V(C_2)$, we are done. If, in turn, $a'\geq3$, then, by Lemma~\ref{lem:4}, $D'$ contains an $M_1$-compatible cycle of length at least $a'$.

Let $C_2$ be a cycle in $D'$ of maximal lenght among all cycles compatible with some complete matching from $X\cap V(D')$ to $Y\cap V(D')$, and let $M_2$ be a complete matching from $X\cap V(D')$ to $Y\cap V(D')$ with which $C_2$ is compatible. If $V(D')=V(C_2)$ or $|V(D')\setminus V(C_2)|\leq2$, then we define $M$ to coincide with $M_2$ on $V(C_2)$ and with $M_1$ on $V(D)\setminus V(C_2)$, and the construction is complete. Otherwise, we set $D''=D-(V(C_1)\cup V(C_2))$ and $a''=|D''|/2$. We have $a''\geq2$, hence we can apply Lemma~\ref{lem:5} to $D''$, to get that $D''$ satisfies condition $\A$. We can thus apply Lemma~\ref{lem:4} to $D''$. If $a''=2$, then Lemma~\ref{lem:4} implies that $D''$ contains a cycle $C_3$ of length $4$, which defines a complete matching $M_3$ from $X\cap V(D'')$ to $Y\cap V(D'')$. Setting $M$ to coincide with $M_1$ on $V(C_1)$, with $M_2$ on $V(C_2)$, and with $M_3$ on $V(C_3)$, we are done. If, in turn, $a''\geq3$, then, by Lemma~\ref{lem:4}, $D''$ contains an $M_2$-compatible cycle of length at least $a''$.

We can choose now a cycle $C_3$ in $D''$ of maximal lenght among all cycles compatible with some complete matching from $X\cap V(D'')$ to $Y\cap V(D'')$, and let $M_3$ be a complete matching from $X\cap V(D'')$ to $Y\cap V(D'')$ with which $C_3$ is compatible. We can continue the above procedure until the remaining set of vertices is empty or of cardinality $2$, as required.\medskip

Having constructed the matching $M$ and cycles $C_1,\dots,C_k$ as above, let us introduce the following notation and terminology. For $j=1,\dots,k$, set $D_j=D[V(C_j)]$, and set $D_{k+1}=D[V^r]$ provided $V^r\neq\varnothing$. For convenience, we will call $D_1,\dots,D_k,D_{k+1}$ the \emph{components} of $D$. (Of course, $D$ is connected, by Lemma~\ref{lem:3}, so this terminology should cause no confusion.) Let $c_j=|D_j|/2$, $j=1,\dots,k$. Further, let $R_1=D$ and, for $j=2,\dots,k$, let $R_j=D-(V(C_1)\cup\dots\cup V(C_{j-1}))$, and $R_{k+1}=D_{k+1}$ provided $V^r\neq\varnothing$. Set $a_j=|R_j|/2$, $j=1,\dots,k$. Then, by construction,
\begin{equation}
\label{eq:c_j}
a_j\geq c_j\geq\frac{a_j}{2}\quad\text{for\ }j=1,\dots,k,\qquad\text{and}\qquad 2\leq c_j\leq c_{j-1}\quad\text{for\ }j=2,\dots,k\,.
\end{equation}
\smallskip

Next, we shall prove the following:
\subsection{Claim.}
\label{subsec:claim}
If $V^r\neq\varnothing$, then $R_k$ contains an arc from $Y\cap V(D_k)$ to $V^r$ or an arc from $V^r$ to $X\cap V(D_k)$.
Moreover, for every $m=0,\dots,k-1$, either
\begin{itemize}
\item[(1)] there exists $0\leq l\leq m-1$ such that $R_{k-l}$ consists of at least two components of $D$ and contains an $M$-compatible path $P_{k-l}$ with the following properties: the initial and terminal vertex of $P_{k-l}$ are in $D_{k-l}$, $P_{k-l}$ contains no other vertices of $D_{k-l}$, and $A(D_j)\cap A(P_{k-l})\neq\varnothing$ for every component $D_j$ in $R_{k-(l-1)}$; or else
\item[(2)] $R_{k-(m-1)}$ contains an $M$-compatible path $P_{k-(m-1)}$ with the following properties: precisely one endpoint of $P_{k-(m-1)}$ lies in $D_{k-(m-1)}$, $P_{k-(m-1)}$ contains no other vertices of $D_{k-(m-1)}$, $A(D_j)\cap A(P_{k-(m-1)})\neq\varnothing$ for every component $D_j$ in $R_{k-(m-2)}$, and $P_{k-(m-1)}$ cannot be extended to an $M$-compatible path with both endpoints in $D_{k-(m-1)}$ and $A(D_{k-(m-1)})\cap A(P_{k-(m-1)})=\varnothing$. In this case, $P_{k-(m-1)}$ can be extended to an $M$-compatible path with one endpoint in $D_{k-m}$. (In case when $R_{k-(m-2)}=\varnothing$, $P_{k-(m-1)}$ consists of a single vertex.)
\end{itemize}
Finally, if $m=k-1$ and $R_2=R_{k-(m-1)}$ satisfies condition (2) above, then the path $P_2$ can be extended to an $M$-compatible path with both ends in $D_1$.\medskip

We will proceed by induction on $m$. First, suppose that $V^r\neq\varnothing$ and $R_k$ contains no arc from $Y\cap V(D_k)$ to $V^r$ nor from $V^r$ to $X\cap V(D_k)$. Write $V^r=\{u,v\}$, where $u\in X$ and $v\in Y$. Then we have $d_{D_k}^-(u)=d_{D_k}^+(v)=0$, hence $d_{R_k}^-(u)\leq1$ and $d_{R_k}^+(v)\leq1$. Let $x'\in X\cap V(D_k)$ and $y'\in Y\cap V(D_k)$ be arbitrary. By Lemma~\ref{lem:5}, $R_k$ satisfies condition $\A$, and hence
\begin{multline}
\notag
6a_k+2\leq d_{R_k}(u)+d_{R_k}(v)+d_{R_k}(x')+d_{R_k}(y')\\
=(d_{R_k}^-(u)+d_{R_k}^+(v))+(d_{R_k}^+(u)+d_{R_k}^-(v)+d_{R_k}^+(x')+d_{R_k}^-(y'))+(d_{R_k}^-(x')+d_{R_k}^+(y'))\\
\leq 2+4a_k+(d_{R_k}^-(x')+d_{R_k}^+(y'))\,.
\end{multline}
Consequently, $d_{R_k}^-(x')+d_{R_k}^+(y')\geq2a_k=2c_k+2$, and so $R_k$ contains the arcs $vx'$, $y'u$; a contradiction. This proves the first statement of Claim~\ref{subsec:claim}, as well as establishes the basis for induction in case $V^r\neq\varnothing$.

If, in turn, $V^r=\varnothing$, then $k\geq2$ (as $D$ is not hamiltonian, by hypothesis) and it suffices to show that $R_{k-1}$ contains an arc from $Y\cap V(D_{k-1})$ to $X\cap V(D_k)$ or from $Y\cap V(D_k)$ to $X\cap V(D_{k-1})$. Suppose otherwise. Then $d_{D_{k-1}}^-(x)=d_{D_{k-1}}^+(y)=0$ for all $x\in X\cap V(D_k)$ and $y\in Y\cap V(D_k)$. By Lemma~\ref{lem:5}, $R_{k-1}$ satisfies condition $\A$, and hence, for any pairwise disjoint $x',x''\in X\cap V(D_k)$ and $y',y''\in Y\cap V(D_k)$, we have
\begin{multline}
\notag
6a_{k-1}+2\leq d_{R_{k-1}}(x')+d_{R_{k-1}}(y')+d_{R_{k-1}}(x'')+d_{R_{k-1}}(y'')\\
=(d_{R_{k-1}}^-(x')+d_{R_{k-1}}^+(y')+d_{R_{k-1}}^-(x'')+d_{R_{k-1}}^+(y''))\\
+(d_{R_{k-1}}^+(x')+d_{R_{k-1}}^-(y')+d_{R_{k-1}}^+(x'')+d_{R_{k-1}}^-(y''))\\
=(d_{D_k}^-(x')+d_{D_k}^+(y')+d_{D_k}^-(x'')+d_{D_k}^+(y''))
+(d_{R_{k-1}}^+(x')+d_{R_{k-1}}^-(y')+d_{R_{k-1}}^+(x'')+d_{R_{k-1}}^-(y''))\\
\leq 4c_k+4a_{k-1}\,.
\end{multline}
Consequently, $2c_k\geq a_{k-1}+1=c_{k-1}+c_k+1$, which contradicts \eqref{eq:c_j}.
\medskip

Suppose now that $R_{k-m}$ does not satisfy condition (1) of Claim~\ref{subsec:claim}. Then, by the inductive hypothesis, $R_{k-(m-1)}$ contains an $M$-compatible path $P_{k-(m-1)}$ with the following properties: precisely one endpoint of $P_{k-(m-1)}$ lies in $D_{k-(m-1)}$, $P_{k-(m-1)}$ contains no other vertices of $D_{k-(m-1)}$, $A(P_{k-(m-1)})\cap A(D_j)\neq\varnothing$ for every component $D_j$ in $R_{k-(m-2)}$, and $P_{k-(m-1)}$ cannot be extended to an $M$-compatible path with both endpoints in $D_{k-(m-1)}$ and $A(D_{k-(m-1)})\cap A(P_{k-(m-1)})=\varnothing$.
As for the orientation of $P_{k-(m-1)}$, there are two possibilities: either its initial point lies in $D_{k-(m-1)}$ and the terminal point lies in $D_{k-l}$ for some $l<m-1$, or the initial point lies in $D_{k-l}$ for some $l<m-1$ and the terminal point lies in $D_{k-(m-1)}$. The argument in both case is virtually the same, so we will assume, without loss of generality, that the former is the case.

We will show that $P_{k-(m-1)}$ can be extended to an $M$-compatible path with one endpoint in $D_{k-m}$. Suppose otherwise. Then
\begin{equation}
\label{eq:no-extension}
\begin{aligned}
d&_{D_{k-m}}^-(x)=0\quad\text{for\ all\ }x\in X\cap V(D_{k-(m-1)})\,,\\
d&_{D_{k-m}}^+(y)=0\quad\text{for\ all\ }y\in Y\cap V(D_{k-l})\,,\quad\text{and}\\
d&_{D_{k-(m-1)}}^+(y)=0\quad\text{for\ all\ }y\in Y\cap V(D_{k-l})\,,
\end{aligned}
\end{equation}
where the first (resp. second) line in \eqref{eq:no-extension} follows from the fact that $P_{k-(m-1)}$ cannot be a extended to an $M$-compatible path with the initial (resp. terminal) vertex in $D_{k-m}$, and the last line follows from the assumption that $P_{k-(m-1)}$ cannot be extended to an $M$-compatible path with both endpoints in $D_{k-(m-1)}$.

We claim that then there exists $y^*\in Y\cap V(D_{k-(m-1)})$ such that $d_{D_{k-m}}^+(y^*)>0$. Suppose otherwise. Since, by Lemma~\ref{lem:5}, $R_{k-m}$ satisfies condition $\A$, then, by \eqref{eq:no-extension}, for any pairwise disjoint $x',x''\in X\cap V(D_{k-(m-1)})$ and $y',y''\in Y\cap V(D_{k-(m-1)})$, we get
\begin{multline}
\notag
6a_{k-m}+2\leq d_{R_{k-m}}(x')+d_{R_{k-m}}(y')+d_{R_{k-m}}(x'')+d_{R_{k-m}}(y'')\\
=(d_{R_{k-(m-1)}}^-(x')+d_{R_{k-(m-1)}}^+(y')+d_{R_{k-(m-1)}}^-(x'')+d_{R_{k-(m-1)}}^+(y''))\\
+(d_{R_{k-m}}^+(x')+d_{R_{k-m}}^-(y')+d_{R_{k-m}}^+(x'')+d_{R_{k-m}}^-(y''))\\
\leq 4a_{k-(m-1)}+4a_{k-m}\,,
\end{multline}
hence $a_{k-m}+1\leq 2a_{k-(m-1)}=2(a_{k-m}-c_{k-m})$, and so $c_{k-m}\leq(a_{k-m}-1)/2$, which contradicts \eqref{eq:c_j}.

It follows that there exist distinct $x',x''\in X\cap V(D_{k-(m-1)})$, $y'\in Y\cap V(D_{k-m})$ and $y''\in Y\cap V(D_{k-l})$ such that $R_{k-m}$ contains $M$-compatible paths from $x'$ to $y'$ and from $x''$ to $y''$. Indeed, if $y^*\in Y\cap V(D_{k-(m-1)})$ is as above, then there exists $x^*\in X\cap V(D_{k-m})$ such that $y^*x^*\in A(R_{k-m})$, and this connection from $D_{k-(m-1)}$ to $D_{k-m}$ allows one to construct the first path. A path from $x''$ to $y''$ is constructed by first following $C_{k-(m-1)}$ from $x''$ to the initial point of $P_{k-(m-1)}$, then following $P_{k-(m-1)}$ until it reaches $C_{k-l}$, and then following $C_{k-l}$ until $y''$. By condition $\A$ and \eqref{eq:no-extension} again, we obtain
\begin{multline}
\notag
6a_{k-m}+2\leq d_{R_{k-m}}(x')+d_{R_{k-m}}(y')+d_{R_{k-m}}(x'')+d_{R_{k-m}}(y'')\\
=(d_{R_{k-(m-1)}-V(D_{k-l})}^-(x')+d_{R_{k-(m-1)}-V(D_{k-l})}^-(x''))+d_{R_{k-m}-V(D_{k-(m-1)})}^+(y')\\
+d_{R_{k-(m-2)}}^+(y'')+(d_{R_{k-m}}^+(x')+d_{R_{k-m}}^-(y')+d_{R_{k-m}}^+(x'')+d_{R_{k-m}}^-(y''))\\
\leq 2(a_{k-(m-1)}-c_{k-l})+(a_{k-m}-c_{k-(m-1)})+a_{k-(m-2)}+4a_{k-m}\\
=2a_{k-(m-1)}-2c_{k-l}+a_{k-m}-c_{k-(m-1)}+(a_{k-m}-c_{k-m}-c_{k-(m-1)})+4a_{k-m}\,,
\end{multline}
hence
\begin{multline}
\label{eq:last-line}
2\leq 2a_{k-(m-1)}-2c_{k-l}-2c_{k-(m-1)}-c_{k-m}\\
=(a_{k-(m-1)}-2c_{k-(m-1)})+(a_{k-(m-1)}-c_{k-m})-2c_{k-l}\,.
\end{multline}
By \eqref{eq:c_j}, both $a_{k-(m-1)}-2c_{k-(m-1)}$ and $a_{k-(m-1)}-c_{k-m}$ are at most $0$, and so the content of the second line of \eqref{eq:last-line} is negative; a contradiction.

To complete the proof of Claim~\ref{subsec:claim}, it remains to show that if $R_2$ satisfies condition (2) of the claim, then the path $P_2$ can be extended to an $M$-compatible path with both ends in $D_1$. As above, without loss of generality, assume that the initial vertex of $P_2$ lies in $D_2$ and its terminal vertex lies in $D_{k-l}$ for some $l\leq k-2$ (in case when $D$ consists of precisely two components, we have $D_{k-l}=D_2$ and hence $P_2$ is a single vertex). We have already established that $D=R_1$ contains at least one of the following: an arc from $Y\cap V(D_1)$ to $X\cap V(D_2)$, or an arc from $Y\cap V(D_{k-l})$ to $X\cap V(D_1)$. We want to show that, in fact, $D$ contains both kinds of arcs. Suppose otherwise; say, suppose $D$ does not contain an arc from $Y\cap V(D_{k-l})$ to $X\cap V(D_1)$. If $|D_{k-l}|>2$, then, for any distinct $y',y''\in Y\cap V(D_{k-l})$, we have $d_D^+(y')=d_{R_2}^+(y')$ and $d_D^+(y'')=d_{R_2}^+(y'')$. Therefore, since $D$ satisfies condition $\M$, we get
\begin{multline}
\notag
3a+1\leq d_D(y')+d_D(y'')=(d_{R_2}^+(y')+d_{R_2}^+(y''))+(d_D^-(y')+d_D^-(y''))\\
\leq 2(a-c_1)+2a\,,
\end{multline}
hence $c_1\leq(a-1)/2$, which contradicts \eqref{eq:c_j}. If, in turn, $|D_{k-l}|=2$, then the sole vertex of $Y\cap V(D_{k-l})$ must dominate a vertex of $X\cap V(D_1)$, by Remark~\ref{rem:non-hamiltonian} and because condition (1) of Claim~\ref{subsec:claim} does not hold in $R_2$. The proof in the case when $D$ contains no arc from $Y\cap V(D_1)$ to $X\cap V(D_2)$ is analogous.

\subsection{Extending a maximal cycle. Ooops...}
We will now complete the proof of Theorem~\ref{thm:main} by showing that the path from Claim~\ref{subsec:claim} can be used to extend one of the maximal cycles constructed in \ref{subsec:cycles}, thus contradicting its maximality.

By Claim~\ref{subsec:claim}, we can choose $m\in\{0,\dots,k-1\}$ such that $R_{k-m}$ consists of at least two components of $D$ and contains an $M$-compatible path $P'$ with the following properties:
\begin{itemize}
\item[(1)] the initial and terminal vertex of $P'$ are in $D_{k-m}$,
\item[(2)] $P'$ contains no other vertices of $D_{k-m}$,
\item[(3)] $A(P')\cap A(D_j)\neq\varnothing$ for every component $D_j$ in $R_{k-(m-1)}$.
\item[(4)] If, moreover, $R_{k-(m-1)}$ contains at least two components of $D$, then $P'-V(D_{k-m})$ could not be extended to an $M$-compatible cycle in $R_{k-(m-1)}$.
\end{itemize}
By $M$-compatibility, the initial vertex of $P'$ belongs to $Y$ and its terminal vertex belongs to $X$. Let $P$ be the path obtained from $P'$ by removing these two endpoints, and write $P=(u,\dots,v)$, where $u\in X$ and $v\in Y$. Write $C_{k-m}=[x_1,y_1,\dots,x_{c_{k-m}},y_{c_{k-m}}]$, according to the $M$-compatible orientation of $C_{k-m}$. Then, there exist $y_i$ and $x_j$ in $C_{k-m}$ such that $y_iu,vx_j\in A(D)$. Choose $i_0,j_0\in\{1,\dots,c_{k-m}\}$ such that $y_{i_0}u,vx_{j_0}\in A(D)$, and, if $P^{i_0j_0}$ is the path from $x_{i_0+1}$ to $y_{j_0-1}$ on $C_{k-m}$, then $y_\nu u\notin A(D)$ and $vx_\nu\notin A(D)$ for all $y_\nu\in Y\cap V(P^{i_0j_0})$ and $x_\nu\in X\cap V(P^{i_0j_0})$. Set $\mu=|P^{i_0j_0}|/2$. Of course, $\mu\geq1$, for else $C_{k-m}$ could be extended to a strictly longer $M$-compatible cycle by replacing the arc $y_{i_0}x_{j_0}$ in $C_{k-m}$ with the path $(y_{i_0},u,\dots,v,x_{j_0})$.

By condition (4) above, $u$ and $v$ belong to different components of $D$, unless (i) $m=0$ and $V^r\neq\varnothing$, or (ii) $m=1$ and $V^r=\varnothing$. Suppose first that neither (i) nor (ii) hold. 
Note that
\begin{equation}
\label{eq:u-v}
d_{D_{k-m}}^-(u)+d_{D_{k-m}}^+(v)\leq(c_{k-m}-\mu-1)+2=c_{k-m}-\mu+1\,,
\end{equation}
for else $C_{k-m}$ would contain consecutive vertices $y_s\in Y\cap V(D_{k-m})$ and $x_{s+1}\in X\cap V(D_{k-m})$ such that $y_su,vx_{s+1}\in A(D)$. Consequently, $R_{k-m}$ would contain an $M$-compatible cycle $[x_{s+1},y_{s+1},\dots,y_s,u,\dots,v]$ of length strictly greater than $|C_{k-m}|$, contradicting the maximality of $C_{k-m}$ in $R_{k-m}$.

Next, observe that $x_{i_0+1}$ (the successor of $y_{i_0}$ on $C_{k-m}$) and $y_{j_0-1}$ (the predecessor of $x_{j_0}$ on $C_{k-m}$) satisfy
\begin{equation}
\label{eq:co-u-v}
d_{R_{k-(m-1)}}^-(x_{i_0+1})=d_{R_{k-(m-1)}}^+(y_{j_0-1})=0\,.
\end{equation}
Indeed, for if, for example, $x_{i_0+1}$ were dominated by a vertex $y^*$ from one of the components of $R_{k-(m-1)}$, then one could replace the arc $y_{i_0}x_{i_0+1}$ in $C_{k-m}$ with an $M$-compatible path $(y_{i_0},u,\dots,y^*,x_{i_0+1})$. (The fact that every such $y^*$ lies on an $M$-compatible path starting at $u$ follows from condition (3) above.)

Now, by Lemma~\ref{lem:5}, $R_{k-m}$ satisfies condition $\A$, and hence, by \eqref{eq:u-v} and \eqref{eq:co-u-v},
\begin{multline}
\label{eq:key1}
6a_{k-m}+2\leq d_{R_{k-m}}(x_{i_0+1})+d_{R_{k-m}}(y_{j_0-1})+d_{R_{k-m}}(u)+d_{R_{k-m}}(v)\\
=(d_{R_{k-m}}^-(x_{i_0+1})+d_{R_{k-m}}^+(y_{j_0-1}))+(d_{R_{k-m}}^-(u)+d_{R_{k-m}}^+(v))\\
+(d_{R_{k-m}}^+(x_{i_0+1})+d_{R_{k-m}}^-(y_{j_0-1})+d_{R_{k-m}}^+(u)+d_{R_{k-m}}^-(v))\\
=(d_{D_{k-m}}^-(x_{i_0+1})+d_{D_{k-m}}^+(y_{j_0-1}))+(d_{D_{k-m}}^-(u)+d_{D_{k-m}}^+(v))+(d_{R_{k-(m-1)}}^-(u)+d_{R_{k-(m-1)}}^+(v))\\
+(d_{R_{k-m}}^+(x_{i_0+1})+d_{R_{k-m}}^-(y_{j_0-1})+d_{R_{k-m}}^+(u)+d_{R_{k-m}}^-(v))\\
\leq (d_{D_{k-m}}^-(x_{i_0+1})+d_{D_{k-m}}^+(y_{j_0-1}))+(c_{k-m}-\mu+1)+2a_{k-(m-1)}+4a_{k-m}\,.
\end{multline}
Therefore,
\begin{multline}
\label{eq:key2}
d_{D_{k-m}}^-(x_{i_0+1})+d_{D_{k-m}}^+(y_{j_0-1})\geq 2(a_{k-m}-a_{k-(m-1)})-c_{k-m}+\mu+1\\
=2c_{k-m}-c_{k-m}+\mu+1=c_{k-m}+\mu+1\,.
\end{multline}

If the inequality in \eqref{eq:key2} is strict, then $C_{k-m}$ contains consecutive vertices $y_s\in Y\cap V(D_{k-m})$ and $x_{s+1}\in X\cap V(D_{k-m})$ such that $y_{j_0-1}x_{s+1},y_sx_{i_0+1}\in A(D)$. Then $R_{k-m}$ contains an $M$-compatible cycle
\[
[x_{s+1},\dots,y_{i_0},u,\dots,v,x_{j_0},\dots,y_s,x_{i_0+1},\dots,y_{j_0-1}]
\]
of length strictly greater than $|C_{k-m}|$, which contradicts the maximality of $C_{k-m}$ in $R_{k-m}$. If, in turn, the two sides of \eqref{eq:key2} are equal, then we also have equality in \eqref{eq:key1}. In particular, $d_{R_{k-(m-1)}}^+(v)=a_{k-(m-1)}$, hence $vu\in A(D)$. Then $P'-V(D_{k-m})$ can be extended to an $M$-compatible cycle in $R_{k-(m-1)}$, contradicting condition (4) above.
\medskip

To complete the proof of Theorem~\ref{thm:main}, it remains to consider the cases when $m=0$ and $V^r\neq\varnothing$, or $m=1$ and $V^r=\varnothing$. If $m=0$ and $V^r\neq\varnothing$, then $u,v\in V^r$, and hence $P$ is, in fact, the arc $uv$ itself. Same as above, the strict inequality in \eqref{eq:key2} implies that $C_k=C_{k-m}$ can be extended to an $M$-compatible cycle of strictly greater length, which contradicts the choice of $C_k$. Therefore, both sides of \eqref{eq:key2} (hence also of \eqref{eq:key1}) are equal. In particular, $d_{R_k}^+(u)=d_{R_k}^-(v)=a_k$. It follows that $u$ dominates every vertex of $Y\cap V(C_k)$, $v$ is dominated by every vertex of $X\cap V(C_k)$, and $vu\in A(D)$. Thus, $R_k$ contains, for example, a cycle $C'=[x_1,v,u,y_1,x_2,\dots,y_{c_k}]$ of length $c_k+1$. Since $C'$ passes through all the vertices of $R_k$, it defines a complete matching from $X\cap V(R_k)$ to $Y\cap V(R_k)$. But $|C'|>|C_k|$, which contradicts the choice of $C_k$ (condition (ii) in \ref{subsec:cycles}).

Finally, suppose that $m=1$ and $V^r=\varnothing$. Then $u,v\in V(C_k)$. As above, we must have equality in \eqref{eq:key2}, hence also in \eqref{eq:key1}. In particular, $d_{R_{k-m}}^+(x_{i_0+1})=a_{k-m}$, $d_{R_{k-(m-1)}}^-(u)=a_{k-(m-1)}$, and $d_{R_{k-m}}^+(u)=a_{k-m}$.
Let $v'\in Y\cap V(D_k)$ denote the successor of $u$ on $C_k$ (it may be that $v'=v$). Then the above degree conditions imply that the arcs $x_{i_0+1}v'$, $v'u$ and $uy_{i_0+1}$ are all in $A(D)$. Consequently, $R_{k-1}$ contains a cycle $C''=[x_{i_0+1},v',u,y_{i_0+1},\dots,y_{i_0}]$ of length $c_{k-1}+1$. By assumption, $C_{k-1}$ is compatible with a complete matching $M$ from $X\cap V(R_{k-1})$ to $Y\cap V(R_{k-1})$. However, one can define a new complete matching $M'$ as follows: let $x_{i_0+1}v',uy_{i_0+1}\in M'$ and let $M'$ coincide with $M$ on $V(R_{k-1})\setminus\{x_{i_0+1},y_{i_0+1},u,v'\}$. Then $C''$ is a cycle in $R_{k-1}$ compatible with a complete matching from $X\cap V(R_{k-1})$ to $Y\cap V(R_{k-1})$ (namely, $M'$) and of length strictly greater than $c_{k-1}$, which contradicts the choice of $C_{k-1}$ (condition (ii) in \ref{subsec:cycles}). \qed

\medskip
\bibliographystyle{amsplain}

\end{document}